\documentclass[12pt]{article}
\usepackage{amsfonts}
\usepackage{amsmath}
\usepackage{amsthm}

\newtheorem{definition}{Definition}[section]
\newtheorem{theorem}{Theorem}[section]
\newtheorem{corollary}{Corollary}[section]
\newcommand{\supp}{\textup{supp}}
\numberwithin{equation}{section}

\begin{document}
\title{Majorization results for zeros of orthogonal polynomials \thanks{This research was supported by KU Leuven research grant OT/12/073 and FWO research project G.0864.16N.}}
\author{Walter Van Assche \\ KU Leuven}
%\address{Department of Mathematics, KU Leuven,
%Celestijnenlaan 200B box 2400, BE-3001 Leuven, Belgium}
%\email{walter@wis.kuleuven.be}
\date{\today}
%\subjclass[2010]{Primary 33C45, 42C05; Secondary 26C10, 30C15, 15B51}
%\keywords{Orthogonal polynomials, zeros, doubly stochastic matrices}
\maketitle

\begin{abstract}
We show that the zeros of consecutive orthogonal polynomials $p_n$ and $p_{n-1}$ are linearly connected by a doubly stochastic matrix for which the entries are explicitly computed in terms of Christoffel numbers.
We give similar results for the zeros of $p_n$ and the associated polynomial $p_{n-1}^{(1)}$ and for
the zeros of the polynomial obtained by deleting the $k$th row and column $(1 \leq k \leq n)$ in the corresponding Jacobi matrix.
\end{abstract}

\section{Introduction}

\subsection{Majorization}

Recently S.M. Malamud \cite{Malamud0,Malamud1} and R. Pereira \cite{Pereira} have given a nice extension of the
Gauss-Lucas theorem that the zeros of the derivative $p'$ of a polynomial $p$ lie in the convex hull of the roots of $p$.
They both used the theory of majorization of sequences to show that if $x=(x_1,x_2,\ldots,x_n)$ are the zeros of a polynomial $p$ of degree $n$, and $y=(y_1,\ldots,y_{n-1})$ are the zeros of its derivative $p'$, then there exists a doubly stochastic
$(n-1) \times n$ matrix $S$ such that $y = Sx$ (\cite[Thm. 4.7]{Malamud1}, \cite[Thm. 5.4]{Pereira}). A rectangular $(n-1) \times n$ matrix $S = (s_{i,j})$ is said to be doubly stochastic if $s_{i,j} \geq 0$ and
\[   \sum_{j=1}^n s_{i,j} = 1, \quad \sum_{i=1}^{n-1} s_{i,j} = \frac{n-1}{n}. \]  
In this paper we will use the notion of majorization to rephrase some known results for the zeros of orthogonal polynomials on the real line.
It is an open problem to see what can be said about the zeros of orthogonal polynomials in the complex plane in terms of majorization.

Let us begin by defining the notion of majorization. An excellent source of information on majorization is the book by Marshall and Olkin (and Arnold, who was added for the second edition) \cite{Marshall}. See also Horn and Johnson \cite[Def. 4.3.24]{HornJohnson}.
\begin{definition}
let $x=(x_1,x_2,\ldots,x_n) \in \mathbb{R}^n$ and $y=(y_1,y_2,\ldots,y_n) \in \mathbb{R}^n$. Then the vector $y$ is said to majorize the vector $x$, 
or $x$ is majorized by $y$ 
(notation: $x \prec y$) if
\[    \hat{x}_1+ \cdots + \hat{x}_j \leq \hat{y}_1 + \cdots + \hat{y}_j, \qquad 1 \leq j \leq n-1, \]
and
\[   \hat{x}_1+ \cdots + \hat{x}_n = \hat{y}_1 + \cdots + \hat{y}_n,  \]
where $\hat{x} = (\hat{x}_1,\ldots,\hat{x}_n)$ and $\hat{y} = (\hat{y}_1,\ldots,\hat{y}_n)$ contain the components of $x$ and $y$ in decreasing order. 
\end{definition}

The following result (of Hardy, Littlewood and P\'olya) gives an alternative way to define majorization in terms of doubly stochastic matrices.
Recall that a square matrix $A = (a_{i,j})_{i,j=1}^n$ of order $n$ is doubly stochastic if $a_{i,j} \geq 0$ and
\[   \sum_{i=1}^n a_{i,j} = 1, \quad \sum_{j=1}^n a_{i,j} = 1, \]
i.e., the elements in every row and every column add up to $1$.

\begin{theorem}
Let $x = (x_1,\ldots,x_n)$ and $y = (y_1,\ldots,y_n)$. Then $x \prec y$ if and only if there exists a doubly stochastic matrix $A$ such that $x=Ay$.
\end{theorem}

See \cite[Thm. 2.B.2]{Marshall},  \cite[Thm. 4.3.33]{HornJohnson}. Note that the doubly stochastic matrix $A$ need not be unique.

\subsection{Orthogonal polynomials}
Let $\mu$ be a positive measure on the real line for which all the moments exist. To simplify notation, we will assume that $\mu$ is a probability
measure so that $\mu(\mathbb{R})=1$.
The orthogonal polynomials $p_n$ for this measure
satisfy the orthogonality relations
\[   \int p_n(x)p_m(x)\, d\mu(x) = \delta_{m,n}, \]
and we denote the orthonormal polynomials by $p_n(x) = \gamma_n x^n +\cdots$, with the convention that $\gamma_n > 0$.
Orthogonal polynomials on the real line always satisfy a three term recurrence relation
\begin{equation}  \label{3term}
   xp_n(x) = a_{n+1} p_{n+1}(x) + b_n p_n(x) + a_n p_{n-1}(x), 
\end{equation}
with $a_n >0$ for $n \geq 1$ and $b_n \in \mathbb{R}$ for $n \geq 0$. 
These recurrence coefficients are given by
\begin{eqnarray}
     a_n &=& \int xp_n(x)p_{n-1}(x)\, d\mu(x) = \frac{\gamma_{n-1}}{\gamma_n} ,  \label{an} \\
     b_n &=& \int x p_n^2(x)\, d\mu(x).   \label{bn}
\end{eqnarray}
The zeros of $p_n$ are all real and simple and we will denote them by
$x_{1,n} < x_{2,n} < \cdots < x_{n,n}$, hence in increasing order. It is well known that the zeros of $p_{n-1}$ and $p_n$ interlace, i.e.,
\[      x_{j,n} < x_{j,n-1} < x_{j+1,n}, \qquad 1 \leq j \leq n-1.  \]
The zeros of $p_n$ are also the eigenvalues of the tridiagonal matrix (Jacobi matrix)
\begin{equation}  \label{jacobi}
      J_n = \begin{pmatrix}   b_0 & a_1 & 0   & 0     & \cdots &   0  \\
                                a_1 & b_1 & a_1 & 0   & \cdots &   0 \\
                                 0  & a_2 & b_2 & a_3 &  &   0  \\
                                 \vdots  &  &  \ddots & \ddots & \ddots &  \vdots \\
                                 0  & \cdots & 0 & a_{n-2} & b_{n-2} & a_{n-1} \\
                                 0  & \cdots & 0& 0 &   a_{n-1} & b_{n-1} 
             \end{pmatrix}                
\end{equation}
and the interlacing of the zeros of $p_n$ and $p_{n-1}$ hence corresponds to the interlacing of the eigenvalues of $J_n$ and $J_{n-1}$.
Another consequence is that for the trace of $J_n$ one has
\[   \textrm{Tr } J_n = x_{1,n}+x_{2,n} + \cdots x_{n,n} = b_0 + b_1 + \cdots + b_{n-1}, \]
which gives a relation between the sum of the zeros of $p_n$ and a partial sum of the recurrence coefficients $(b_n)_{n \geq 0}$.
In particular we find
\begin{equation}  \label{sumxb}
      x_{1,n-1}+x_{2,n-1} + \cdots x_{n-1,n-1} + b_{n-1} = x_{1,n} + x_{2,n} + \cdots + x_{n,n}.  
\end{equation} 
The interlacing of the zeros of $p_{n-1}$ and $p_n$, together with \eqref{sumxb} then imply that the vector 
$(x_{1,n-1},\ldots,x_{n-1,n-1},b_{n-1})$ is majorized
by $(x_{1,n},x_{2,n}, \ldots, x_{n,n})$. In Section \ref{sec2} we will give an explicit expression for a doubly stochastic matrix $A$ for which
\[   \begin{pmatrix}  x_{1,n-1} \\ x_{2,n-1} \\ \vdots \\ x_{n-1,n-1} \\ b_{n-1} \end{pmatrix}
                    = A \begin{pmatrix}  x_{1,n} \\ x_{2,n} \\ \vdots \\ x_{n-1,n} \\ x_{n,n}  \end{pmatrix} .  \]

Another interesting interlacing property involves the zeros of the associated polynomial 
\[   p_{n-1}^{(1)}(z) = a_1 \int \frac{p_n(z)-p_n(x)}{z-x}\, d\mu(x),  \]
see, e.g., \cite{WVA}.
This polynomial appears naturally as the numerator of the Pad\'e approximant to the function
\[    f(z) = a_1 \int \frac{d\mu(x)}{z-x} = a_1 \sum_{k=0}^\infty \frac{m_k}{z^{k+1}}, \qquad m_k = \int x^k\, d\mu(x), \]
near infinity, i.e.,
\[     p_n(z) f(z) - p_{n-1}^{(1)}(z) =  \mathcal{O}(1/z^{n+1}), \qquad z \to \infty. \] 
The zeros of $p_{n-1}^{(1)}$ are again real and simple and we denote them by  $y_{1,n-1} < y_{2,n-1} < \cdots < y_{n-1,n-1}$. 
They interlace with the zeros of $p_n$ in the sense that
\[     x_{j,n} < y_{j,n-1} < x_{j+1,n}, \qquad  1 \leq j \leq n-1.  \]
In fact, the partial fractions decomposition of the Pad\'e approximant is
\begin{equation}  \label{parfrac}
   \frac{p_{n-1}^{(1)}(z)}{p_n(z)} = a_1 \sum_{k=1}^n \frac{\lambda_{k,n}}{z-x_{k,n}}, 
\end{equation}
where the $\lambda_{k,n}$ are the Christoffel numbers 
\begin{equation}  \label{lambda}
   \lambda_{k,n} = \frac{1}{\sum_{j=0}^{n-1} p_j^2(x_{k,n})} = \frac{1}{a_n p_n'(x_{k,n})p_{n-1}(x_{k,n})} .
\end{equation} 
The zeros of $p_{n-1}^{(1)}$ are equal to the eigenvalues of the Jacobi matrix $J_n^{(1)}$ obtained from $J_n$ by deleting the first row and first column:
\begin{equation}  \label{jacobi2}
       J_n^{(1)} = \begin{pmatrix}   b_1 & a_2 & 0   & 0     & \cdots &   0  \\
                                a_2 & b_2 & a_3 & 0   & \cdots &   0 \\
                                 0  & a_3 & b_3 & a_4 &  &   0  \\
                                 \vdots  &  &  \ddots & \ddots & \ddots &  \vdots \\
                                 0  & \cdots & 0 & a_{n-2} & b_{n-2} & a_{n-1} \\
                                 0  & \cdots & 0& 0 &   a_{n-1} & b_{n-1} 
             \end{pmatrix}   . 
\end{equation}             
For the trace we therefore have
\[   \textrm{Tr } J_n^{(1)} = y_{1,n-1} + y_{2,n-1} + \cdots + y_{n-1,n-1} = b_1+b_2+\cdots+b_{n-1}, \]
and hence
\begin{equation}  \label{sumyb}
     y_{1,n-1} + y_{2,n-1} + \cdots + y_{n-1,n-1} + b_0 = x_{1,n}+x_{2,n} + \cdots + x_{n,n} . 
\end{equation}
Combining the interlacing property and \eqref{sumyb} we then see that $(y_{1,n-1},\ldots,y_{n-1,n-1},b_0)$ is majorized by
$(x_{1,n},x_{2,n},\ldots,x_{n,n})$. In Section \ref{sec3} we will give an explicit expression for a doubly stochastic matrix $B$ for which
\[   \begin{pmatrix}  y_{1,n-1} \\ y_{2,n-1} \\ \vdots \\ y_{n-1,n-1} \\ b_{0} \end{pmatrix}
                    = B \begin{pmatrix}  x_{1,n} \\ x_{2,n} \\ \vdots \\ x_{n-1,n} \\ x_{n,n}  \end{pmatrix} .  \]
Finally, in Section \ref{sec4} we will prove a similar result for the zeros of the characteristic polynomial of the matrix obtained by deleting the $k$th row and
column of the Jacobi matrix $J_n$.

We will frequently use Gaussian quadrature using the zeros of orthogonal polynomials. Suppose that  $\textrm{supp}(\mu) \subset [a,b]$ and that
$\{x_{j,n}, 1 \leq j \leq n\}$ are the zeros of the
orthogonal polynomial $p_n$ and $\{\lambda_{j,n}, 1 \leq j \leq n\}$ the corresponding Christoffel numbers, then for every function
$f \in C^{2n}([a,b])$
\begin{equation}  \label{gauss}
    \int f(x) \, d\mu(x) = \sum_{j=1}^n \lambda_{j,n} p(x_{j,n}) + \frac{f^{(2n)}(\xi)}{(2n)! \gamma_n^2}, 
\end{equation}
where $\gamma_n$ is the leading coefficient of $p_n$ and $\xi \in (a,b)$ (see, e.g., \cite[\S 1.4.2]{Gautschi}, \cite[\S 3.4]{Szego}).
The quadrature sum therefore gives the integral exactly whenever $f$ is a polynomial of degree $\leq 2n-1$. The Christoffel numbers $\lambda_{j,n}$
given in \eqref{lambda} are also given by
\begin{equation}  \label{lambdaint} 
  \lambda_{j,n} = \int_a^b \frac{p_n^2(x)}{(x-x_{j,n})^2 [p_n'(x_{j,n}]^2}\, d\mu(x) > 0. 
\end{equation}  

\section{Zeros of consecutive orthogonal polynomials}   \label{sec2}

\begin{theorem}  \label{thm:2.1}
Let $(x_{1,n-1},\ldots,x_{n-1,n-1})$ be the zeros of $p_{n-1}$ and $(x_{1,n},\ldots,x_{n,n})$ be the zeros of $p_n$, 
and let $(a_{n+1},b_n)_{n \geq 0}$ be the recurrence coefficients in \eqref{3term} and \eqref{jacobi}. Then  
\begin{equation}  \label{x=Ax}
  \begin{pmatrix}  x_{1,n-1} \\ x_{2,n-1} \\ \vdots \\ x_{n-1,n-1} \\ b_{n-1} \end{pmatrix}
                    = A \begin{pmatrix}  x_{1,n} \\ x_{2,n} \\ \vdots \\ x_{n-1,n} \\ x_{n,n}  \end{pmatrix} ,  
\end{equation}
where $A = (a_{i,j})_{i,j=1}^n$ is a doubly stochastic matrix with entries
\begin{equation}  \label{A}
   a_{i,j} = \begin{cases}
\displaystyle \frac{a_n^2 \lambda_{j,n} p_{n-1}^2(x_{j,n})\lambda_{k,n-1} p_{n}^2(x_{i,n-1})}{(x_{j,n}-x_{i,n-1})^2}, & 1 \leq i \leq n-1, \\[10pt]
              \lambda_{j,n} p_{n-1}^2(x_{j,n}), & i = n. \end{cases}  
\end{equation}
\end{theorem}

\begin{proof}
Let us consider the integral
\[    \int_a^b x \frac{p_{n-1}^2(x)}{(x-x_{i,n-1})^2}\, d\mu(x). \]
By adding and subtracting, we easily find
\begin{multline*}
  \int_a^b x \frac{p_{n-1}^2(x)}{(x-x_{i,n-1})^2}\, d\mu(x) \\
    = \int_a^b (x-x_{k,n}) \frac{p_{n-1}^2(x)}{(x-x_{i,n-1})^2}\, d\mu(x)
    + x_{i,n-1} \int_a^b \frac{p_{n-1}^2(x)}{(x-x_{i,n-1})^2}\, d\mu(x), 
\end{multline*}
and the first integral on the right vanishes by orthogonality since $p_{n-1}$ is orthogonal to the polynomial $p_{n-1}(x)/(x-x_{i,n-1})$ of
degree $n-2$. Hence
\begin{equation}  \label{xkn-1}
   x_{i,n-1} = \frac{\displaystyle \int_a^b x \frac{p_{n-1}^2(x)}{(x-x_{i,n-1})^2}\, d\mu(x)}
               {\displaystyle \int_a^b \frac{p_{n-1}^2(x)}{(x-x_{i,n-1})^2}\, d\mu(x)} . 
\end{equation}
Observe that by \eqref{lambdaint} (for $n-1$) the denominator is equal to
\[   \int_a^b  \frac{p_{n-1}^2(x)}{(x-x_{i,n-1})^2}\, d\mu(x) = \lambda_{i,n-1} [p_{n-1}'(x_{i,n-1})]^2 . \]
If we use \eqref{lambda} (for $n-1$) then
\begin{equation}  \label{lambda2}
    \lambda_{i,n-1} = \frac{1}{a_{n-1} p_{n-1}'(x_{i,n-1})p_{n-2}(x_{i,n-1})} = \frac{-1}{a_n p_{n-1}'(x_{i,n-1})p_n(x_{i,n-1})}, 
\end{equation}
where we used the recurrence relation \eqref{3term} which gives $a_{n-1}p_{n-2}(x_{i,n-1}) = - a_n p_n(x_{i,n-1})$. Therefore
the denominator is given by
\[   \int_a^b  \frac{p_{n-1}^2(x)}{(x-x_{i,n-1})^2}\, d\mu(x) = \frac{1}{a_n^2 \lambda_{i,n-1} p_{n}^2(x_{i,n-1})} . \]
For the numerator, we use the Gaussian quadrature rule with the zeros of $p_n$ to find
\[     \int_a^b x \frac{p_{n-1}^2(x)}{(x-x_{i,n-1})^2}\, d\mu(x) = \sum_{j=1}^n \lambda_{j,n} x_{j,n} \frac{p_{n-1}^2(x_{j,n})}{(x_{j,n}-x_{i,n-1})^2}, \]
so that \eqref{xkn-1} gives
\[    x_{i,n-1} = \sum_{j=1}^n a_{i,j} x_{j,n}, \qquad 1 \leq i \leq n-1, \]
with $a_{i,j}$ given in \eqref{A} for $1 \leq i \leq n-1$. Furthermore, we have
\[   b_{n-1} = \int_a^b x p_{n-1}^2(x)\, d\mu(x) = \sum_{j=1}^n \lambda_{j,n} x_{j,n} p_{n-1}^2(x_{j,n}) = \sum_{j=1}^n a_{n,j} x_{j,n}, \]
where we used \eqref{bn} and Gaussian quadrature. So \eqref{A} also holds for $i=n$.
It remains to show that the matrix $A$ is doubly stochastic. Clearly $a_{i,j} > 0$ because the Christoffel numbers are positive.
For the row sums we have
\[   \sum_{j=1}^n a_{i,j} = a_n^2 \lambda_{i,n-1} p_{n}^2(x_{i,n-1}) \sum_{j=1}^n \lambda_{j,n} \frac{p_{n-1}^2(x_{j,n})}{(x_{j,n}-x_{i,n-1})^2},
\]
and by the Gaussian quadrature rule, this is
\[   a_n^2 \lambda_{i,n-1} p_{n}^2(x_{i,n-1}) \int_a^b \frac{p_{n-1}^2(x)}{(x-x_{i,n-1})^2}\, d\mu(x) 
 =  a_n^2 \lambda_{i,n-1}^2 p_n^2(x_{i,n-1}) [p_{n-1}'(x_{i,n-1})]^2 = 1, \]
where we used \eqref{lambdaint} for $n-1$ and \eqref{lambda2}.
For the column sums we have
\[   \sum_{i=1}^{n-1} a_{i,j} = a_n^2 \lambda_{j,n} p_{n-1}^2(x_{j,n}) \sum_{i=1}^{n-1} \lambda_{i,n-1} \frac{p_n^2(x_{i,n-1})}{(x_{j,n}-x_{i,n-1})^2} . \]
By Gaussian quadrature (with the zeros of $p_{n-1}$ and remainder) we thus find
\[  \sum_{i=1}^{n-1} a_{i,j} = a_n^2 \lambda_{j,n} p_{n-1}^2(x_{j,n}) \left( \int_a^b \frac{p_n^2(x)}{(x-x_{j,n})^2}\, d\mu(x) 
- \frac{1}{a_n^2} \right), \]
because the remainder in \eqref{gauss} is
\[   \frac{1}{(2n-2)! \gamma_{n-1}^2} \frac{d^{2n-2}}{dx^{2n-2}} \left( \gamma_n x^{2n-2} + \cdots \right) 
    = \frac{\gamma_{n}^2}{\gamma_{n-1}^2} = \frac{1}{a_n^2}, \]
where we used \eqref{an}. Together with \eqref{lambdaint} and \eqref{lambda} we then find
\[   \sum_{i=1}^{n-1} a_{i,j} = 1 - \lambda_{j,n} p_{n-1}^2(x_{j,n}) = 1 - a_{n,j}, \]
proving that the column sums add to 1.
\end{proof}

\begin{corollary}
Suppose $\supp(\mu) \subset [a,b]$ and $f$ is a convex function on $[a,b]$, then
\[  \sum_{i=1}^{n-1} f(x_{i,n-1}) + f(b_{n-1})  \leq \sum_{j=1}^n f(x_{j,n}) .  \] 
\end{corollary}

\begin{proof}
The relation \eqref{x=Ax} gives
\[    x_{i,n-1} = \sum_{j=1}^n a_{i,j} x_{j,n}, \]
and since the row sum adds to one and $a_{i,j} >0$, this means that the right hand side is a convex combination of the zeros $x_{j,n}$. For a convex function
one has
\[     f(\lambda x + (1-\lambda)y) \leq \lambda f(x) + (1-\lambda) f(y), \qquad 0 < \lambda < 1, \]
hence for this convex combination
\[    f(x_{i,n-1}) \leq \sum_{j=1}^n a_{i,j} f(x_{j,n}), \qquad 1 \leq i \leq n-1.   \]
The same reasoning also gives
\[   f(b_{n-1}) \leq \sum_{j=1}^n a_{n,j} f(x_{j,n}).  \]
Adding over $i$ and using that the column sums of $A$ add to 1, then gives the required result.
\end{proof}

\section{Zeros of associated polynomials}  \label{sec3}
 \begin{theorem}  \label{thm:3.1}
Let $(y_{1,n-1},\ldots,y_{n-1,n-1})$ be the zeros of $p_{n-1}^{(1)}$ and $(x_{1,n},\ldots,x_{n,n})$ be the zeros of $p_n$, 
and let $(a_{n+1},b_n)_{n \geq 0}$ be the recurrence coefficients in \eqref{3term} and \eqref{jacobi2}. Then  
\begin{equation}  \label{y=Bx}
  \begin{pmatrix}  y_{1,n-1} \\ y_{2,n-1} \\ \vdots \\ y_{n-1,n-1} \\ b_{0} \end{pmatrix}
                    = B \begin{pmatrix}  x_{1,n} \\ x_{2,n} \\ \vdots \\ x_{n-1,n} \\ x_{n,n}  \end{pmatrix} ,  
\end{equation}
where $B = (b_{i,j})_{i,j=1}^n$ is a doubly stochastic matrix with entries
\begin{equation}  \label{B}
   b_{i,j} = \begin{cases}
\displaystyle a_1^2 \frac{\lambda_{j,n}\lambda_{i,n-1}^{(1)}}{(y_{i,n-1}-x_{j,n})^2}, & 1 \leq i \leq n-1, \\[10pt]
              \lambda_{j,n} , & i = n, \end{cases}  
\end{equation}
where $\lambda_{i,n-1}^{(1)}$ are the Christoffel numbers of the polynomials $p_{n-1}^{(1)}$ for the measure $\mu^{(1)}$.
\end{theorem}

\begin{proof}
If we use the partial fractions decomposition \eqref{parfrac}, then
\[   \frac{p_{n-1}^{(1)}(z)}{p_n(z)} = a_1 \sum_{j=1}^n \frac{\lambda_{j,n}(z-x_{j,n})}{(z-x_{j,n})^2} .  \]
If we take $z=y_{i,n-1}$ then this expression vanishes, hence
\begin{equation}  \label{ykn-1}
    y_{i,n-1} \sum_{j=1}^n \frac{\lambda_{j,n}}{(y_{i,n-1}-x_{j,n})^2} = \sum_{j=1}^n x_{j,n} \frac{\lambda_{j,n}}{(y_{i,n-1}-x_{j,n})^2} . 
\end{equation}
Observe that
\[   a_1 \sum_{j=1}^n \frac{\lambda_{j,n}}{(z-x_{j,n})^2} = - \frac{d}{dz} \frac{p_{n-1}^{(1)}(z)}{p_n(z)}
   = - \frac{[p_{n-1}^{(1)}(z)]'p_n(z) - p_n'(z)p_{n-1}^{(1)}(z)}{p_n^2(z)} . \]
Taking $z=y_{i,n-1}$ then gives
\begin{equation}   \label{lambda/(x-y)2}
   \sum_{j=1}^n \frac{\lambda_{j,n}}{(y_{i,n-1}-x_{j,n})^2} = \frac{-1}{a_1} \frac{[p_{n-1}^{(1)}(y_{i,n-1})]'}{p_n(y_{i,n-1})}, 
\end{equation}  
so that \eqref{ykn-1} gives
\[    y_{i,n-1} = \sum_{j=1}^n b_{i,j} x_{j,n}, \]
with
\[   b_{i,j} = - a_1 \frac{\lambda_{j,n}}{(y_{i,n-1}-x_{j,n})^2} \frac{p_n(y_{i,n-1})}{[p_{n-1}^{(1)}(y_{i,n-1})]'}.  \]
The associated polynomials are again orthogonal polynomials, but now with respect to another measure $\mu^{(1)}$. The Christoffel
numbers $\lambda_{i,n-1}^{(1)}$ for this measure are obtained by shifting all the recurrence coefficients up by one, and hence
\eqref{lambda} gives
\begin{equation}  \label{lambda(1)}
   \lambda_{i,n-1}^{(1)} = \frac{1}{a_n[p_{n-1}^{(1)}(y_{i,n-1})]'p_{n-2}^{(1)}(y_{i,n-1})} 
                         = \frac{-1}{a_{n+1} [p_{n-1}^{(1)}(y_{i,n-1})]' p_n^{(1)}(y_{i,n-1})}, 
\end{equation}
which gives
\[   b_{i,j} = a_1 \frac{\lambda_{j,n} \lambda_{i,n-1}^{(1)} a_{n+1} p_n(y_{i,n-1})p_n^{(1)}(y_{i,n-1})}{(y_{i,n-1}-x_{j,n})^2}, \]
and if we use the Wronskian formula \cite[Eq. (2.3)]{WVA}
\begin{equation}   \label{Wronskian}
   a_{n+1} \Bigl( p_n(x)p_n^{(1)}(x) - p_{n+1}(x) p_{n-1}^{(1)}(x) \Bigr) = a_1, 
\end{equation}
then for $x=y_{i,n-1}$ this gives $a_{n+1} p_n(y_{i,n-1})p_n^{(1)}(y_{i,n-1}) = a_1$, from which the formula \eqref{B} follows
for $1 \leq i \leq n-1$. Note that from \eqref{bn} and Gaussian quadrature
\[   b_0  =  \int x \, d\mu(x) =  \sum_{j=1}^n \lambda_{j,n} x_{j,n}, \]
so that \eqref{B} is also true for $i=n$.

We will now show that the matrix $B = (b_{i,j})_{1\leq i,j \leq n}$ is doubly stochastic. Obviously $b_{i,j} >0$ because the
Christoffel numbers are all positive. For the row sums we have by \eqref{lambda/(x-y)2}
\[   \sum_{j=1}^n b_{i,j} = a_1^2 \lambda_{i,n-1}^{(1)} \sum_{j=1}^n \frac{\lambda_{j,n}}{(y_{i,n-1}-x_{j,n})^2}
    = -a_1 \lambda_{i,n-1}^{(1)}  \frac{[p_{n-1}^{(1)}(y_{i,n-1})]'}{p_n(y_{i,n-1})} = 1, \]
where the last equality follows from $a_{n+1} p_n(y_{i,n-1})p_n^{(1)}(y_{i,n-1}) = a_1$ and \eqref{lambda(1)}.
For the column sums we have
\[   \sum_{i=1}^{n-1} b_{i,j} = a_1^2 \lambda_{j,n} \sum_{i=1}^{n-1} \frac{\lambda_{i,n-1}^{(1)}}{(y_{i,n-1}-x_{j,n})^2} . \]
Now use the partial fractions decomposition
\[   \frac{p_{n-2}^{(2)}(z)}{p_{n-1}^{(1)}(z)} = a_2 \sum_{i=1}^{n-1} \frac{\lambda_{i,n-1}^{(1)}}{z-y_{i,n-1}}. \]
From the identity \cite[Eq. (2.5)]{WVA} we have
\[     zp_{n-1}^{(1)}(z) = a_1 p_n(z) + b_0 p_{n-1}^{(1)}(z) + \frac{a_1^2}{a_2} p_{n-2}^{(2)}(z), \]
so that
\[   \frac{a_1^2}{a_2} \frac{p_{n-2}^{(2)}(z)}{p_{n-1}^{(1)}(z)} = z-b_0 - a_1 \frac{p_n(z)}{p_{n-1}^{(1)}(z)} , \] and hence 
\[   a_1^2 \sum_{i=1}^{n-1} \frac{\lambda_{i,n-1}^{(1)}}{z-y_{i,n-1}} = z-b_0 - a_1 \frac{p_n(z)}{p_{n-1}^{(1)}(z)} . \]
Taking the derivative then gives
\[  - a_1^2 \sum_{i=1}^{n-1} \frac{\lambda_{i,n-1}^{(1)}}{(z-y_{i,n-1})^2} 
   = 1 - a_1 \frac{p_n'(z)p_{n-1}^{(1)}(z)-p_n(z)[p_{n-1}^{(1)}(z)]'}{[p_{n-1}^{(1)}(z)]^2} , \]
which after evaluation at $z=x_{j,n}$ gives
\[   -a_1^2 \sum_{i=1}^{n-1} \frac{\lambda_{i,n-1}^{(1)}}{(x_{j,n}-y_{i,n-1})^2} 
   = 1 - a_1 \frac{p_n'(x_{j,n})}{p_{n-1}^{(1)}(x_{j,n})} . \]
Therefore
\[  \sum_{i=1}^{n-1}  b_{i,j} = \lambda_{j,n} \left( a_1 \frac{p_n'(x_{j,n})}{p_{n-1}^{(1)}(x_{j,n})} - 1 \right). \]
The Wronskian formula \eqref{Wronskian} evaluated at $x_{j,n}$ gives $-a_{n+1}p_{n+1}(x_{j,n})p_{n-1}^{(1)}(x_{j,n}) = a_1$, so that
\[  \sum_{i=1}^{n-1}  b_{i,j} = \lambda_{j,n} \bigl( -a_{n+1}p_{n+1}(x_{j,n})p_n'(x_{j,n}) - 1 \bigr) = 1 - \lambda_{j,n}, \]
where we used \eqref{lambda} to get the last equality. This shows that the column sums also add to one.
\end{proof}

\begin{corollary}
Suppose $\supp(\mu) \subset [a,b]$ and  $f$ is a convex function on $[a,b]$, then
\[  \sum_{i=1}^{n-1} f(y_{i,n-1}) + f(b_0)  \leq \sum_{j=1}^n f(x_{j,n}) .  \] 
\end{corollary}

\section{Deleting a row and column in the Jacobi matrix} \label{sec4}
The result in Theorem \ref{thm:2.1} corresponds to the eigenvalues of the Jacobi matrix $J_n$ after deleting the last row and column, which gives the Jacobi matrix $J_{n-1}$. In a similar way, the result in Theorem \ref{thm:3.1} corresponds to deleting the first row and column of the Jacobi matrix, giving the Jacobi matrix $J_n^{(1)}$. A natural question now is to ask what would happen when one deletes the $k$th row and column of the Jacobi matrix $J_n$, with $1 \leq k \leq n$. The resulting Jacobi matrix is
\[    \begin{pmatrix} b_0 & a_1 & 0   & 0 & 0 & 0 & 0 & 0 & \cdots & 0 \\
                      a_1 & b_1 & a_1 & 0 & 0 & 0 & 0 & 0 & \cdots & 0\\
                      0 & a_2 & \ddots & \ddots &  &  &  &  & \vdots & \vdots \\
                      \vdots  &  & \ddots & b_{k-3} & a_{k-2} & 0 & 0 & 0 &\cdots &  0 \\
                      0  & \cdots  & 0 & a_{k-2} & b_{k-2} & 0 & 0 & 0 & \cdots  & 0 \\
                      0  & \cdots  & 0 & 0 &  0 & b_{k} & a_{k+1} & 0 & \cdots  &  0 \\
                      0   & \cdots  & 0 & 0 & 0  &  a_{k+1} & b_{k+1} & a_{k+2} &  & \vdots \\
                      0  & \cdots  & 0 & 0 & 0  &      0    & a_{k+2} & \ddots & \ddots &  0\\
                      \vdots  &  \vdots  &  &  &   &          &         & \ddots & b_{n-2} & a_{n-1} \\
                      0  & \cdots  & 0 & 0 & 0  &   0       &    0     &   0    & a_{n-1} & b_{n-1}
       \end{pmatrix}   . \]
Observe that this matrix consists of the Jacobi matrix $J_{k-1}$ and the associated Jacobi matrix $J_{n-k}^{(k)}$ obtained by shifting the coefficients by $k$, hence the characteristic polynomial of this matrix is $c p_{k-1}(z)p_{n-k}^{(k)}(z)$, where $c$ is a constant that makes this polynomial
monic, i.e., $c = \prod_{j=1,j\neq k}^{n} a_j$. From \cite[Eq. (3.5)]{WVA} we have the partial fractions decomposition
\begin{equation}    \label{parfrack}
   \frac{p_{k-1}(z)p_{n-k}^{(k)}(z)}{p_n(z)} = a_k  \sum_{j=1}^n \frac{\lambda_{j,n}p_{k-1}^2(x_{j,n})}{z-x_{j,n}}.  
\end{equation}
We can now formulate and prove the following extension of Theorems \ref{thm:2.1} and \ref{thm:3.1}.

\begin{theorem}  \label{thm:4.1}
Let $(x_{1,k-1},\ldots,x_{k-1,k-1})$ be the zeros of $p_{k-1}$, and $(y_{k,n-k},\ldots,y_{n-1,n-k})$ be the zeros of $p_{n-k}^{(k)}$.
Then
\begin{equation}   \label{xy=Cx}
    \begin{pmatrix} x_{1,k-1} \\ \vdots \\ x_{k-1,k-1} \\ y_{k,n-k} \\ \vdots \\  y_{n-1,n-k} \\ b_{k-1} \end{pmatrix}
     = C   \begin{pmatrix}  x_{1,n} \\ x_{2,n} \\ x_{3,n} \\ \vdots \\ x_{n-2,n} \\ x_{n-1,n} \\ x_{n,n}  \end{pmatrix},  
\end{equation}
where $C = (c_{i,j})_{i,j=1}^n$ is a doubly stochastic matrix with entries
\begin{equation}   \label{C}
     c_{i,j} = \begin{cases}
      \displaystyle  a_k^2 \frac{\lambda_{j,n} \lambda_{i,k-1} p_{k-1}^2(x_{j,n})p_k^2(x_{i,k-1})}{(x_{i,k-1}-x_{j,n})^2}, & 1 \leq i \leq k-1, \\[10pt]
      \displaystyle  a_k^2 \frac{\lambda_{j,n} \lambda_{i,n-k}^{(k)} p_{k-1}^2(x_{j,n})}{(y_{i,n-k}-x_{j,n})^2}, & k \leq i \leq n-1, \\[10pt]
                  \lambda_{j,n} p_{k-1}^2(x_{j,n}), & i = n,  \end{cases}
\end{equation}
where $(\lambda_{1,n},\ldots,\lambda_{n,n})$ are the Christoffel numbers of the orthogonal polynomials $p_n$ for the measure $\mu$ 
and $(\lambda_{k,n-k}^{(k)},\ldots,\lambda_{n-1,n-k}^{(k)})$ are the Christoffel numbers of the polynomials $p_{n-k}^{(k)}$ for the measure $\mu^{(k)}$. 
\end{theorem}

Observe that for $k=1$ we retrieve Theorem \ref{thm:3.1} and $k=n$ corresponds to Theorem \ref{thm:2.1}.

\begin{proof}
From \eqref{parfrack} we easily find
\begin{equation}  \label{pfk}
    \frac{p_{k-1}(z)p_{n-k}^{(k)}(z)}{p_n(z)} = a_k  \sum_{j=1}^n \frac{\lambda_{j,n}p_{k-1}^2(x_{j,n})(z-x_{j,n})}{(z-x_{j,n})^2}, 
\end{equation}
and by putting $z=x_{i,k-1}$ we obtain
\[     x_{i,k-1} \sum_{j=1}^n \frac{\lambda_{j,n}p_{k-1}^2(x_{j,n})}{(x_{i,k-1}-x_{j,n})^2} 
  = \sum_{j=1}^n x_{j,n} \frac{\lambda_{j,n}p_{k-1}^2(x_{j,n})}{(x_{i,k-1}-x_{j,n})^2} ,  \]
which gives $x_{i,k-1} = \sum_{j=1}^n c_{i,j} x_{j,n}$, with
\[    c_{i,j} = \frac{\lambda_{j,n} p_{k-1}^2(x_{j,n})/(x_{i,k-1}-x_{j,n})^2}{\sum_{\ell=1}^n \frac{\lambda_{\ell,n}p_{k-1}^2(x_{\ell,n})}{(
x_{i,k-1}-x_{\ell,n})^2}}.  \]
From this it is already clear that $c_{i,j} \geq 0$ and $\sum_{j=1}^n c_{i,j} = 1$, and this holds for $1 \leq i \leq k-1$.
In a similar way we find, by taking $z=y_{i,n-k}$ in \eqref{pfk}, 
\[    y_{i,n-k} \sum_{j=1}^n \frac{\lambda_{j,n}p_{k-1}^2(x_{j,n})}{(y_{i,n-k}-x_{j,n})^2} 
  = \sum_{j=1}^n x_{j,n} \frac{\lambda_{j,n}p_{k-1}^2(x_{j,n})}{(y_{i,n-k}-x_{j,n})^2} ,  \]
so that $y_{i,n-k} = \sum_{j=1}^n c_{i,j} x_{j,n}$, with
\[    c_{i,j} = \frac{\lambda_{j,n} p_{k-1}^2(x_{j,n})/(y_{i,n-k}-x_{j,n})^2}{\sum_{\ell=1}^n \frac{\lambda_{\ell,n}p_{k-1}^2(x_{\ell,n})}{(
y_{i,n-k}-x_{\ell,n})^2}},  \]
for $k \leq i \leq n-1$. Again it follows immediately that $c_{i,j} \geq 0$ and $\sum_{j=1}^n c_{i,j} = 1$ for $k \leq i \leq n-1$.
For $i=n$ we use \eqref{bn} and Gaussian quadrature to find
\[     b_{k-1} = \int x p_{k-1}^2(x)\, d\mu(x) = \sum_{j=1}^n \lambda_{j,n} x_{j,n} p_{k-1}^2(x_{j,n}),  \]
so that $c_{n,j} = \lambda_{j,n} p_{k-1}^2(x_{j,n})$. Clearly $c_{n,j} \geq 0$ and $\sum_{j=1}^n c_{n,j} = 1$ follows using Gaussian quadrature.

It remains to show that the $c_{i,j}$ can be simplified to \eqref{C} and that the sums in every column add to 1.
Taking the derivative in \eqref{parfrack} gives
\[   \frac{d}{dz} \frac{p_{k-1}(z)p_{n-k}^{(k)}(z)}{p_n(z)} = - a_k \sum_{\ell=1}^n \frac{\lambda_{\ell,n} p_{k-1}^2(x_{\ell,n})}{(z-x_{\ell,n})^2}, \]
and if we put $z=x_{i,k-1}$, for which $p_{k-1}(x_{i,k-1})=0$,  then this gives
\[    \frac{p_{k-1}'(x_{i,k-1})p_{n-k}^{(k)}(x_{i,k-1})}{p_n(x_{i,k-1})} = -a_k \sum_{\ell=1}^n \frac{\lambda_{\ell,n}p_{k-1}^2(x_{\ell,n})}
         {(x_{i,k-1}-x_{\ell,n})^2} . \]
In a similar way we put $z=y_{i,n-k}$, for which $p_{n-k}^{(k)}(y_{i,n-k})=0$, to find
\[   \frac{p_{k-1}(y_{i,n-k})[p_{n-k}^{(k)}(y_{i,n-k})]'}{p_n(y_{i,n-k})} = -a_k \sum_{\ell=1}^n \frac{\lambda_{\ell,n}p_{k-1}^2(x_{\ell,n})}
         {(y_{i,n-k}-x_{\ell,n})^2} . \]
We thus find
\[  c_{i,j} = \begin{cases}
        \displaystyle  -a_k \frac{\lambda_{j,n} p_{k-1}^2(x_{j,n}) p_n(x_{i,k-1})}{(x_{i,k-1}-x_{j,n})^2 p_{k-1}'(x_{i,k-1})p_{n-k}^{(k)}(x_{i,k-1})},
         &  1 \leq i \leq k-1, \\[20pt]
        \displaystyle  -a_k \frac{\lambda_{j,n} p_{k-1}^2(x_{j,n}) p_n(y_{i,n-k})}{(y_{i,n-k}-x_{j,n})^2 p_{k-1}(y_{i,n-k})[p_{n-k}^{(k)}(y_{i,n-k})]'},
         &  k \leq i \leq n-1.
             \end{cases} \]
By using \eqref{lambda2} one has
\[   \lambda_{i,k-1} = \frac{-1}{a_kp_{k-1}'(x_{i,k-1})p_k(x_{i,k-1})}. \]
From \cite[Eq. (2.4)]{WVA} it follows that
\begin{equation}   \label{W24}
     a_1 p_{n-k}^{(k)}(x) = a_k \Bigl( p_{k-1}(x)p_{n-1}^{(1)}(x) - p_n(x)p_{k-2}^{(1)}(x) \Bigr), 
\end{equation}
and for $x=x_{i,k-1}$ this gives
\[   a_1 p_{n-k}^{(k)}(x_{i,k-1}) = -a_k p_{k-1}^{(1)}(x_{i,k-1}) p_n(x_{i,k-1}). \]
Then, from the Wronskian formula \cite[Eq. (2.3)]{WVA} one has
\[   a_k \Bigl( p_{k-1}(x)p_{k-1}^{(1)}(x) - p_k(x)p_{k-2}^{(1)}(x) \Bigr) = a_1, \]
and for $x=x_{i,k-1}$ this gives
\[    -a_k p_k(x_{i,k-1})p_{k-2}^{(1)}(x_{i,k-1}) = a_1. \]
Combining all this gives
the required expression for $c_{i,j}$ in \eqref{C} for $1 \leq i \leq k-1$.
In a similar way it follows from \eqref{lambda2} that
\[  \lambda_{i,n-k}^{(k)} = \frac{-1}{a_{n+1} [p_{n-k}^{(k)}(y_{i,n-k})]' p_{n-k-1}^{(k)}(y_{i,n-k})} . \]
Then from \cite[Eq. (2.6)]{WVA} we find
\begin{equation}  \label{W26}
   a_{n+1} \Bigl( p_n(x)p_{n-k+1}^{(k)}(x) - p_{n+1}(x) p_{n-k}^{(k)}(x) \Bigr) = a_k p_{k-1}(x), 
\end{equation}
which for $x=y_{i,n-k}$ gives
\[    a_{n+1} p_n(y_{i,n-k}) p_{n-k+1}^{(k)}(y_{i,n-k}) = a_k p_{k-1}(y_{i,n-k}). \]
Combining all this gives the expression  for $c_{i,j}$ in \eqref{C} for $k \leq i \leq n-1$.

Now we show that the column sums are equal to $1$. We have
\begin{equation}  \label{sumcij}
 \sum_{i=1}^n c_{i,j} = \lambda_{j,n} p_{k-1}^2(x_{j,n}) \left( a_k^2 \sum_{i=1}^{k-1} \frac{\lambda_{i,k-1} p_k^2(x_{i,k-1})}{(x_{i,k-1}-x_{j,n})^2}
      + a_k^2 \sum_{i=k}^{n-1} \frac{\lambda_{i,n-k}^{(k)}}{(y_{i,n-k}-x_{j,n})^2} + 1 \right) .
\end{equation}
From the three term recurrence relation \eqref{3term} we find that $a_kp_k(x_{i,k-1})=a_{k-1} p_{k-2}(x_{i,k-1})$ so that
\[   a_k^2 \sum_{i=1}^{k-1} \frac{\lambda_{i,k-1} p_k^2(x_{i,k-1})}{(x_{i,k-1}-x_{j,n})^2}
    =   a_{k-1}^2 \sum_{i=1}^{k-1} \frac{\lambda_{i,k-1} p_{k-2}^2(x_{i,k-1})}{(x_{i,k-1}-x_{j,n})^2}. \]
From \eqref{parfrack} we find (for $k \to n$ and then $n \to k-1$)
\[   \frac{p_{k-2}(z)}{p_{k-1}(z)} = a_{k-1}  \sum_{i=1}^{k-1} \frac{\lambda_{i,k-1} p_{k-2}^2(x_{i,k-1})}{z-x_{i,k-1}}, \]
which after taking the derivative and for $z=x_{j,n}$ gives
\[    \frac{p_{k-2}'(x_{j,n})p_{k-1}(x_{j,n})-p_{k-2}(x_{j,n})p_{k-1}'(x_{j,n})}{p_{k-1}^2(x_{j,n})}
     = - a_{k-1} \sum_{i=1}^{k-1} \frac{\lambda_{i,k-1} p_{k-2}^2(x_{i,k-1})}{(x_{j,n}-x_{i,k-1})^2} .  \]
Hence 
\[  a_k^2 \sum_{i=1}^{k-1} \frac{\lambda_{i,k-1} p_k^2(x_{i,k-1})}{(x_{i,k-1}-x_{j,n})^2}
    = - a_{k-1}      \frac{p_{k-2}'(x_{j,n})p_{k-1}(x_{j,n})-p_{k-2}(x_{j,n})p_{k-1}'(x_{j,n})}{p_{k-1}^2(x_{j,n})} .  \] 
Recall the confluent form of the Christoffel-Darboux formula \cite[Eq. (3.2.4)]{Szego} \cite[Eq. (1.3.22)]{Gautschi}
\begin{equation}  \label{CD}
   \sum_{j=0}^n p_j^2(x) = a_{n+1} \Bigl( p_{n+1}'(x)p_n(x) - p_{n+1}(x)p_n'(x) \Bigr), 
\end{equation}
then this simplifies to
\begin{equation}  \label{sum1}
a_k^2 \sum_{i=1}^{k-1} \frac{\lambda_{i,k-1} p_k^2(x_{i,k-1})}{(x_{i,k-1}-x_{j,n})^2}
   = \frac{1}{p_{k-1}^2(x_{j,n})} \sum_{i=0}^{k-2} p_i^2(x_{j,n}).  
\end{equation}
Next, from \eqref{parfrac} we find (for the associated polynomials)
\[  \frac{p_{n-k-1}^{(k+1)}(z)}{p_{n-k}^{(k)}(z)} = a_{k+1} \sum_{i=k}^{n-1} \frac{\lambda_{i,n-k}^{(k)}}{z-y_{i,n-k}}, \]
which after taking the derivative and for $z=x_{j,n}$ gives
\[  \frac{[p_{n-k-1}^{(k+1)}(x_{j,n})]'p_{n-k}^{(k)}(x_{j,n})-[p_{n-k}^{(k)}(x_{j,n})]'p_{n-k-1}^{(k+1)}(x_{j,n})}{[p_{n-k}^{(k)}(x_{j,n})]^2}
    = - a_{k+1} \sum_{i=k}^{n-1} \frac{\lambda_{i,n-k}^{(k)}}{(x_{j,n}-y_{i,n-k})^2} . \] 
Now use \eqref{W24} to find
\begin{eqnarray*}
   p_{n-k}^{(k)}(x_{j,n}) &=& \frac{a_k}{a_1} p_{k-1}(x_{j,n}) p_{n-1}^{(1)}(x_{j,n}), \\
   p_{n-k-1}^{(k+1)}(x_{j,n}) &=& \frac{a_{k+1}}{a_1} p_{k}(x_{j,n}) p_{n-1}^{(1)}(x_{j,n}), \\
  {} [p_{n-k}^{(k)}(x_{j,n})]' &=& \frac{a_k}{a_1} \Bigl( p_{k-1}'(x_{j,n}) p_{n-1}^{(1)}(x_{j,n}) + p_{k-1}(x_{j,n})[p_{n-1}^{(1)}(x_{j,n})]' 
     - p_n'(x_{j,n})p_{k-2}^{(1)}(x_{j,n}) \Bigr), \\
  {} [p_{n-k-1}^{(k+1)}(x_{j,n})]' &=& \frac{a_{k+1}}{a_1} \Bigl( p_{k}'(x_{j,n}) p_{n-1}^{(1)}(x_{j,n}) + p_{k}(x_{j,n})[p_{n-1}^{(1)}(x_{j,n})]' 
     - p_n'(x_{j,n})p_{k-1}^{(1)}(x_{j,n}) \Bigr), 
\end{eqnarray*}
then, after some calculus, we find
\begin{multline*}
 a_k \sum_{i=k}^{n-1} \frac{\lambda_{i,n-k}^{(k)}}{(x_{j,n}-y_{i,n-k})^2} \\ =  
  \frac{1}{p_{k-1}^2(x_{j,n})} \left( \frac{p_n'(x_{j,n})}{p_{n-1}^{(1)}(x_{j,n})} [p_{k-1}^{(1)}(x_{j,n})p_{k-1}(x_{j,n}) -
   p_{k-1}^{(1)}(x_{j,n})p_k(x_{j,n})] \right. \\
    \left. \phantom{\frac{p_n'(x_{j,n})}{p_{n-1}^{(1)}(x_{j,n})}} - [p_k'(x_{j,n})p_{k-1}(x_{j,n}) - p_{k-1}'(x_{j,n})p_k(x_{j,n})]  \right). 
\end{multline*}
Using the Wronskian \eqref{Wronskian} and the Christoffel-Darboux formula \eqref{CD} gives
\begin{equation}  \label{sum2}
  a_k^2 \sum_{i=k}^{n-1} \frac{\lambda_{i,n-k}^{(k)}}{(x_{j,n}-y_{i,n-k})^2} = 
   \frac{1}{p_{k-1}^2(x_{j,n})} \left( a_1  \frac{p_n'(x_{j,n})}{p_{n-1}^{(1)}(x_{j,n})} - \sum_{i=0}^{k-1} p_i^2(x_{j,n}) \right). 
\end{equation}
From \eqref{parfrac} we find
\[   \frac{p_{n-1}^{(1)}(x_{j,n})}{p_n'(x_{j,n})} = a_1 \lambda_{j,n}, \]
so that inserting \eqref{sum1} and \eqref{sum2} in  \eqref{sumcij} gives  
\[  \sum_{i=1}^n c_{i,j} = 1, \]
which is what we wanted to prove.
\end{proof}

\begin{corollary}
Suppose $\supp(\mu) \subset [a,b]$ and that $f$ is a convex function on $[a,b]$, then
\[  \sum_{i=1}^{k-1} f(x_{i,k-1}) + \sum_{i=k}^{n-1} f(y_{i,n-k}) + f(b_{k-1}) \leq \sum_{j=1}^n f(x_{j,n}).  \]
\end{corollary} 

\section*{Acknowledgment}
This paper was inspired by the plenary talk of Vilmos Totik at the XII International Conference on Approximation and Optimization in
the Caribbean (Havana, Cuba, 5--10 June, 2016). I'd like to thank Vilmos Totik for his encouragement.

\begin{verbatim}
Walter Van Assche
Department of Mathematics
KU Leuven
Celestijnenlaan 200B box 2400
BE-3001 Leuven
BELGIUM
walter@wis.kuleuven.be
\end{verbatim}

\begin{thebibliography}{9}
\bibitem{Gautschi} W. Gautschi,
\textit{Orthogonal Polynomials: Computation and Approximation},
Numerical Mathematics and Scientific Computation, Oxford University Press, Oxford, 2004.
\bibitem{HornJohnson} R.A. Horn, C.R. Johnson,
\textit{Matrix Analysis},
Cambridge University Press, Cambridge, 1985.
\bibitem{Malamud0} S.M. Malamud,
\textit{An analog of the Poincar\'e separation theorem for normal matrices and the Gauss-Lucas theorem},
Funktsional. Anal. i Prilozhen.  \textbf{37}  (2003),  no. 3, 85--88 (in Russian);  
translation in  Funct. Anal. Appl.  \textbf{37}  (2003),  no. 3, 232--235. 
\bibitem{Malamud1} S.M. Malamud,
\textit{Inverse spectral problem for normal matrices and the Gauss-Lucas theorem},
Trans. Amer. Math. Soc. \textbf{357} (2005), nr.~10, 4043--4064.
\bibitem{Marshall} A.W. Marshall, I. Olkin, B.C. Arnold,
\textit{Inequalities: Theory of Majorization and its Applications}, 2nd edition,
Springer Series in Statistics, Springer, New York, 2011. 
\bibitem{Pereira} R. Pereira,
\textit{Differentiators and the geometry of polynomials},
J. Math. Anal. Appl. \textbf{285} (2003), 336--348.
\bibitem{Szego} G. Szeg\H{o},
\textit{Orthogonal Polynomials},
Amer. Math. Soc. Colloq. Publ. \textbf{23}, AMS, Providence, RI, 1939 (fourth edition, 1975).
\bibitem{WVA} W. Van Assche,
\textit{Orthogonal polynomials, associated polynomials and functions of the second kind},
J. Comput. Appl. Math. \textbf{37} (1991), 237--249.
\end{thebibliography}
\end{document}